\documentclass[12pt,A4paper]{amsart}
\usepackage{amsmath,amsthm,amssymb,color}
\usepackage[bookmarksnumbered,colorlinks]{hyperref}
\hypersetup{colorlinks=true,linkcolor=red,anchorcolor=green,citecolor=cyan}

\newtheorem{theorem}{Theorem}[section]
\newtheorem{lemma}[theorem]{Lemma}
\numberwithin{equation}{section}
\newcommand{\N}{\mathbb{N}}

\newcommand{\C}{\mathbb{C}}

\begin{document}

\title[\tiny Hypercyclic subspaces for sequences of finite order differential operators]{Hypercyclic subspaces for sequences of finite order differential operators}

\author[Bernal]{L.~Bernal-Gonz\'alez}
\address[L. Bernal-Gonz\'alez]{\mbox{}\newline \indent Departamento de An\'alisis Matem\'atico \newline \indent Facultad de Matem\'aticas
\newline \indent Instituto de Matem\'aticas de la Universidad de Sevilla (IMUS)
\newline \indent Universidad de Sevilla
\newline \indent Avenida Reina Mercedes s/n, 41012-Sevilla (Spain).}
\email{lbernal@us.es}

\author[Calder\'on]{M.C.~Calder\'on-Moreno}
\address[M.C.~Calder\'on-Moreno]{\mbox{}\newline \indent Departamento de An\'alisis Matem\'atico \newline \indent Facultad de Matem\'aticas
\newline \indent Instituto de Matem\'aticas de la Universidad de Sevilla (IMUS)
\newline \indent Universidad de Sevilla
\newline \indent Avenida Reina Mercedes s/n, 41012-Sevilla (Spain).}
\email{mccm@us.es}

\author[L\'opez-Salazar]{J.~L\'opez-Salazar}
\address[J.~L\'opez-Salazar]{\mbox{}\newline \indent Departamento de Matem\'atica Aplicada a las Tecnolog\'ias  \newline \indent de la Informaci\'on y de las Comunicaciones
\newline \indent Escuela T\'ecnica Superior de Ingenier\'ia y Sistemas de Telecomunicaci\'on
\newline \indent Universidad Polit\'ecnica de Madrid
\newline \indent Nikola Tesla s/n, 28031-Madrid (Spain).}
\email{jeronimo.lopezsalazar@upm.es}

\author[Prado]{J.A.~Prado-Bassas}
\address[J.A.~Prado-Bassas]{\mbox{}\newline \indent Departamento de An\'alisis Matem\'atico
\newline \indent Facultad de Matem\'aticas
\newline \indent Instituto de Matem\'aticas de la Universidad de Sevilla (IMUS)
\newline \indent Universidad de Sevilla
\newline \indent Avenida Reina Mercedes s/n, 41012-Sevilla (Spain).}
\email{bassas@us.es}



\begin{abstract}
It is proved that, if $(P_n)$ is a sequence of polynomials with complex coefficients having unbounded valences and tending to infinity at sufficiently many points, then there is an infinite dimensional closed subspace of entire functions, as well a dense $\mathfrak{c}$-dimensional subspace of entire functions, all of whose nonzero members are hypercyclic for the corresponding sequence $(P_n(D))$ of differential operators. In both cases, the subspace can be chosen so as to contain any prescribed hypercylic function.
\end{abstract}

\maketitle

\section{Introduction.}

Hypercyclicity and lineability are two subjects in functional analysis that have been thoroughly investigated for the last three decades. Hypercyclicity deals with the search for vectors whose orbits under an operator or sequence of operators is dense in the supporting space, while the goal of lineability is to find linear structures inside nonlinear subsets of a vector space. The connection between both theories starts with a result due to Herrero \cite{Her} in the Hilbert space setting asserting that, if the sets of vectors having dense orbit with respect an operator is nonempty, then the family of these vectors contains, except for zero, a dense vector subspace (in fact, such a subspace is invariant under the operator). The result was extended by Bourdon \cite{Bou} and B\`es \cite{Bes} to the complex locally convex case and the real locally convex case, respectively. Finally, Wengenroth \cite{Wen} in 2003 was able to encompass the general case of any topological vector space. A number of analogous assertions for {\it sequences} of operators have been also found (see below).

\vskip 3pt

Let us proceed by fixing some terminology. Let $X$ and $Y$ be two linear topological spaces (in many cases, $X=Y$) and $T_n:X\to Y$ be an operator (that is, a continuous linear mapping) for each $n\in\N$. An element $x_0\in X$ is said to be {\it hypercyclic} or {\it universal} for the sequence $(T_n)$ whenever its orbit $\left\{T_n(x_0): n\in\N\right\}$ under $(T_n)$ is dense in $Y$. The family $(T_n)$ is called {\it hypercyclic} whenever it has a hypercyclic vector. It is plain that, in order that a sequence $(T_n)$ can be hypercyclic, the space $Y$ must be separable. If $T:X\to X$ is an operator, then a vector $x_0\in X$ is said to be {\it hypercyclic} for $T$ provided that it is hypercyclic for the sequence $(T^n)$ of iterates of $T$, i.e., $T^n=T\circ T\circ\cdots\circ T$ ($n$-fold). The operator $T$ is {\it hypercyclic} when there is a hypercyclic vector for $T$. The symbols $HC(T)$ and $HC((T_n))$ will denote, respectively, the set of hypercyclic vectors of an operator $T$ and of a sequence $(T_n)$ of operators. For background on hypercyclity the reader is referred to the survey \cite{Gr1} and the books \cite{bayartM,grosseP}.

\vskip 3pt

Before going on, let us also introduce a number of notions coming from the modern theory of lineability. A subset $A$ of a vector space $X$ is said to be {\it lineable} whenever there is an infinite dimensional vector subspace \,$M\subset X$ such that $M\subset A\cup\{0\}$. If \,$X$ is, in addition, a topological vector space, then \,$A$ is said to be {\it spaceable} if there is a closed infinite-dimensional vector subspace \,$M \subset X$ \,such that \, $M\subset A\cup\{0\}$, while \,$A$ \,is {\it dense-lineable} if there is a dense vector subspace \,$M\subset X$ \,such that \,$M\subset A\cup\{0\}$. Finally, $A$ is called {\it maximal dense-lineable} if there is a dense vector subspace \,$M \subset X$ \,such that \,$\dim(M)=\dim(X)$ \,and \,$M\subset A\cup\{0\}$. For concepts and results on lineability, the reader is referred to the book \cite{ABPS}. The more recent, finer, notion of pointwise lineability will be recalled in Section \ref{SeccionPuntual} below.

\vskip 3pt

Note that under the preceding terminology, the assertion by Bourdon, B\`es, and Wengen\-roth mentioned in the first paragraph tells us that \,$HC(T)$ \,is dense-lineable provided that \,$T$ \,is a hypercyclic operator on \,$X$. If \,$X$ \,is a Banach space, then \,$HC(T)$ \,is even maximal dense-lineable (see \cite{Be5}). It is easy to provide examples showing that these properties are no longer true for hypercyclic {\it sequences} \,$(T_n)$ \,of operators. Nevertheless, some criteria have been given in order that \,$HC((T_n))$ \,can be dense-lineable (see \cite{Be4,BeCa}). But the approach in \cite{Be4,BeCa} only yields dense subspaces having countable dimension and so it does not give maximal dense lineability, because the dimension of an infinite dimensional F-space is at least $\mathfrak{c}$, the cardinality of the continuum (see \cite{Popoola}).

\vskip 3pt

Turning our attention to spaceability, one can say that this is a more delicate question. In 1996, Montes \cite{montes1996} showed that if \,$B:\ell_2\to\ell_2$\, is the backward shift operator on the Hilbert space of square-summable sequences, then, despite the hypercyclicity of $2B$, the set $HC(2B)$ is not spaceable. Moreover, he provided a sufficient criterion for the spaceability of \,$HC(T)$ \,in the setting of Banach spaces. These results have been extended in many directions, so that currently a number of sufficient criteria for spaceability and non-spaceability of $HC(T)$ and $HC((T_n))$ are known, even in the case of Fr\'echet spaces (see \cite{BoMaPe,BoniGr,GonLM,grosseP,LeoM,LeoMu,Men,Petersson}).

\vskip 3pt

In this paper, we are specially interested in differential operators defined on the Fr\'echet space \,$H(\C)$ \,of all entire functions, endowed with the topology of uniform convergence on compacta. An entire function \,$\Phi$ \,is said to be of exponential type whenever there exist positive constants $A$ and $B$ such that $|\Phi(z)|\leq Ae^{B|z|}$ for all $z\in\C$. The class of these functions will be denoted by \,${\mathcal E}$. It is easy to realize (see, for instance, \cite{BeG,Dic}) that if \,$\Phi(z)=\sum_{n=0}^{\infty}a_n z^n$\, is an entire function with exponential type, then the series \,$\Phi(D)=\sum_{n=0}^{\infty}a_n D^n$ defines an operator on $H(\C)$ acting in the following way:
\begin{equation*}
  (\Phi(D)f)(z) = \sum_{n=0}^{\infty}a_n f^{(n)}(z)
\end{equation*}
for each \,$f\in H(\C)$\, and \,$z\in\C$. Hence, \,$\Phi (D)$\, defines, under the latter conditions, a (generally, infinite order) linear differential operator with constant coefficients.

Godefroy and Shapiro proved in \cite{GoS} the hypercyclicity of every non-scalar operator on \,$H(\C)$ \,commuting with the translation operator $\tau_a$ for every $a\in\C$, where $\tau_a f(z)=f(z+a)$ for each \,$f\in H(\C)$ and \,$z\in\C$. It was also shown in \cite{GoS} that an operator $L$ on $H(\C)$ commutes with every translation operator $\tau_a$ if and only if $L$ commutes with the derivative operator \,$D$, and if and only if $L=\Phi(D)$ for some entire function $\Phi\in\mathcal{E}$. In particular, by choosing \,$\Phi(z)=e^{az}$ \,and \,$\Phi(z)=z$, one recovers the hypercyclicity of each nontrivial translation operator and of the derivative operator, that had been respectively obtained by Birkhoff \cite{Bir} and MacLane \cite{Mac}. Petersson \cite{Petersson} proved in 2006 the spaceability of \,$HC(\Phi (D))$ \,if \,$\Phi$ \,is transcendental. The finite order case (that is, the polynomial case) turned out to be more refractory. In fact, it was not until 2010 that Shkarin \cite{shkarin} showed the spaceability of \,$HC(D)$. Finally, Menet \cite{Men} in 2014 was able to prove the spaceability of \,$HC(P(D))$ \,for any nonconstant polynomial $P$.

\vskip 3pt

The aim of this paper is to study lineability properties (specifically, spaceability, maximal dense lineability, and pointwise lineability) of the family of entire functions that are hypercyclic with respect to a sequence of differential operators generated by polynomials having unbounded valences. This will be carried out in Sections \ref{SeccionSpaceable} and \ref{SeccionPuntual}, while Section \ref{ResultadosPrevios} will be devoted to establish a number of auxiliary results on hypercyclicity and lineability.

\section{Preliminary results.}\label{ResultadosPrevios}

First of all, we present the following rather general spaceability criterion for \,$HC((T_n))$, that is due to Bonet, Mart\'inez-Gim\'enez and Peris \cite[Theorem 3.5]{BoMaPe} (see also \cite{Petersson}).

\begin{theorem} \label{Bonet-Martinez-Peris}
Let \,$T_n : X \to Y$ $(n \in  \N)$ \,be a sequence of continuous linear mappings between two separable Fréchet spaces such that \,$X$ \,supports a continuous norm. Suppose that there are respectively dense subsets \,$X_0\subset X$ \,and \,$Y_0\subset Y$, a closed infinite dimensional subspace \,$M_0\subset X$, \,and mappings \,$S_n:Y_0\to X$
\,$(n\in \N)$ \,satisfying the following properties:
\begin{enumerate}
  \item[(i)] $\lim_{n\to\infty}T_n x=0$\, for all \,$x\in X_0$.
  \item[(ii)] $\lim_{n\to\infty}S_n y=0$\, for all \,$y\in Y_0$.
  \item[(iii)] $\lim_{n\to\infty}T_n S_n y=y$\, for all \,$y\in Y_0$.
  \item[(iv)] $\lim_{n\to\infty}T_n x=0$\, for all \,$x\in M_0$.
\end{enumerate}
Then there exists a closed infinite dimensional subspace of \,$X$ \,all of whose nonzero members are hypercyclic vectors for the sequence \,$(T_n)$. In other words, the set \,$HC((T_n))$ \,is spaceable.
\end{theorem}

Second, we recall a sufficient condition for the denseness in \,$H(\C)$ \,of a family of exponential functions. If $w\in\C$, we define $e_w(z)=e^{wz}$ for each $z\in\C$. Recall that $\mathcal{E}$ denotes the class of entire functions of exponential type. We say that a subset $U\subset\C$ is an {\it $\mathcal{E}$-unicity set} whenever the following property holds: if \,$f\in\mathcal{E}$\, and \,$f(z)=0$\, for all \,$z\in U$, then \,$f\equiv 0$. Note that, by the identity principle for holomorphic functions, if \,$f$\, is an arbitrary entire function vanishing on a nonempty open set \,$U$\, (or even just on a set \,$U$\, having some accumulation point in \,$\C$), then \,$f \equiv 0$. This is not necessary for the class \,$\mathcal{E}$; for instance, if \,$U\subset\C$, \,$n(r)$\, is the number of points of \,$U\cap\left\{z\in\C:|z|\leq r\right\}$, and \,$\chi=\limsup_{r\to\infty}\frac{\log n(r)}{\log r}>1$, then \,$U$\, is an \,${\mathcal E}$-unicity set (e.g., \,$U=\left\{n^{1/2}: n\in\N\right\}$, which gives \,$\chi=2$). Indeed, if \,$f\in\mathcal{E}$, \,$f(z)=0$\, for all $z\in U$, and \,$f\not\equiv 0$, the condition \,$\chi>1$\, would imply that the convergence exponent of the sequence of zeros of \,$f$\, is strictly greater than the growth order of \,$f$\,, which is clearly impossible (see, e.g., \cite[p.~15-17]{Boas}).

\vskip 3pt

The next lemma, whose proof is classical, can be found in \cite{bernalprado} (see also \cite{GoS}).

\begin{lemma}
If \,$U$ is an $\mathcal{E}$-unicity set, then the linear span of \,$\{e_w:\,w\in U\}$ \,is dense in \,$H(\C)$.
\end{lemma}

The following assertion, whose proof can be extracted from \cite[Section 7.3]{ABPS}, will be invoked to prove the maximal dense-lineability of the family of hypercyclic entire functions with respect to a sequence of differential operators:

\begin{theorem}\label{A stronger than B}
Assume that \,$X$ is a metrizable separable topological vector space and that \,$A$ \,and \, $B$ \,are subsets of \,$X$ fulfilling the following properties:
\begin{enumerate}
  \item[\rm (i)] $A+B\subset A$.
  \item[\rm (ii)] $A\cap B=\varnothing$.
  \item[\rm (iii)] $A \, \cup \,\{0\}$ \,contains an infinite dimensional vector subspace \,$M$ \,such that \,$\dim(M)=\dim(X)$.
  \item[\rm (iv)] $B$ is dense-lineable.
\end{enumerate}
Then \,$A$ \,is maximal dense-lineable.
\end{theorem}

\section{Fr\'echet spaces of hypercyclic functions.}\label{SeccionSpaceable}

Recall that, for a nonconstant polynomial with complex coefficients \,$P(z)=\sum_{j=m}^{d}c_j z^j$, with \,$m\leq d$ \,and \,$c_m\neq 0\neq c_d$, its valence or multiplicity at the origin is \,$m\in\N\cup\{0\}$, while its degree is \,$d\in\N$.

\vskip 3pt

In this section we are going to consider the following three properties that may or may not be satisfied by a sequence \,$(P_n)$ \,of polynomials:

\vskip 3pt

\begin{enumerate}
  \item[(P)] There is an $\mathcal{E}$-unicity set \,$U\subset\C$ \,such that \,$\lim_{n\to\infty}|P_n(z)|=+\infty$ \,for all \,$z\in U$.
  \item[(Q)] If \,$P_n(z)=\sum_{j=m(n)}^{d(n)}c_{j,n} z^j$, then for each \,$k\in\N$, it holds that
      \begin{equation*}
        \lim_{n \to \infty} m(n)|c_{m(n),n}|^{k/m(n)} = + \infty
      \end{equation*}
      and the sequence $\{c_{k+m(n),n}: \, n\in\N\}$ \,is bounded, where it is understood that \,$c_{j,n} = 0$ \,if \,$j > d(n)$.
  \item[(R)] There is \,$r>0$ \,such that \,$\lim_{n\to\infty}\min\{|P_n(z)| : \, |z|=r\}=+\infty$.
\end{enumerate}

\vskip 3pt

It should be said that properties (P) and (Q) were dealt with in \cite{bernalprado} in order to obtain residuality for \,$HC((\Phi_n(D)))$, but no lineability feature was aimed therein.

\vskip 3pt

Let us analyze the connections among all these properties.

\vskip 3pt

\begin{enumerate}
\item[$\bullet$] (R) implies (P): simply take \,$U=\{z\in\C : |z|=r\}$, which is an $\mathcal{E}$-unicity set.

\item[$\bullet$] (R) does not imply (Q), and, consequently, (P) does not imply (Q) either. For this, let \,$P_n(z)=\frac{z^n}{n^n}+z^{n+1}$ for each $n\geq 1$. On the one hand, given any \,$r>1$, we have
    \begin{equation*}
      \min \{ |P_n (z)| : \, |z|=r\} \geq r^n\cdot\left(r-\frac{1}{n^n}\right) \to +\infty
      \quad \text{as } n\to\infty,
    \end{equation*}
    which entails (R). On the other hand, we have \,$m(n)=n$ \,and \,$c_{m(n),n}=n^{-n}$.
    Hence, \,$\lim_{n\to\infty}m(n)|c_{m(n),n}|^{k/m(n)}=0$ \,for all \,$k\geq 2$, which tells us that (Q) is not satisfied.

\item[$\bullet$] (Q) does not imply (P), and so (Q) does not imply (R) either. Indeed, if $c_n=n^{-n/\log(n+1)}$ \, and \, $P_n(z)=c_nz^n(1 + z)$ \, for each $n\in\N$, then the sequence of polynomials $(P_n)$ \, satisfies (Q), because \,$m(n)=n$, $(c_n)$ is bounded and, for every $k\in\N$, we have
    \begin{equation*}
      \lim_{n\to\infty}m(n)|c_{m(n),n}|^{k/m(n)}
      = \lim_{n\to\infty}n^{1-\frac{k}{\log(n+1)}} = +\infty.
    \end{equation*}
    However, $(P_n)$ does not fulfill (P), because \,$\lim_{n\to\infty}c_n r^n =0$ \,for all \,$r>0$, and so \,$\lim_{n\to\infty}P_n(z)=0$ \,for all \,$z\in\C$.

\item[$\bullet$] (P) does not imply (R). In order to see this, we consider an enumeration \,$(q_n)$ \, of the positive rational numbers and define \,$P_n(z)=z^n(z-q_n)^n$ for each $n\geq 1$. The interval \,$U=(-\infty,-1)$ is an $\mathcal{E}$-unicity set. For every \,$x>1$\, and \,$z=-x\in U$, we have
    \begin{equation*}
      |P_n(z)| = x^n \cdot (x+q_n)^n > x^n \to +\infty   \quad \text{as } n\to\infty.
    \end{equation*}
    Therefore, the sequence $(P_n)$ \,satisfies (P). On the other hand, fixed any \,$r>0$, there exists a subsequence \,$(q_{n_k})$ \, such that \,$0<r-q_{n_k}<1/(2r)$ \,for all \,$k\in\N$. Then
    \begin{equation*}
      |P_{n_k}(r)| = r^{n_k} \cdot (r-q_{n_k})^{n_k} < \frac{1}{2^{n_k}} \to 0
    \quad \text{as } k\to\infty
    \end{equation*}
    and so (R) does not hold.

\item[$\bullet$] There are sequences of polynomials \,$(P_n)$ \,satisfying all three properties (P), (Q), and (R): simply take \,$P_n(z)=c_n z^n$, where \,$(c_n)$ \,is a sequence of complex numbers such that \,$\liminf_{n\to\infty}|c_n|^{1/n}>0$. Contrary to this, none of the properties (P), (Q), or (R) is fulfilled if \,$\lim_{n\to\infty}|c_n|^{1/n^2}=0$.
\end{enumerate}

We are now ready to establish the promised result. As an application of it, notice that all sequences \,$(P_n)$ \,considered in the preceding items satisfy condition (a) of the following theorem. Observe also that the conclusion yields the existence, under appropriate assumptions, of an infinite dimensional Fr\'echet space consisting, except for zero, of $(P_n(D))$-hypercyclic functions.

\begin{theorem} \label{main theorem spaceable}
Let \,$\displaystyle\left\{P_n(z)=\sum_{j=m(n)}^{d(n)}c_{j,n} z^j:n\in\N\right\}$ \,be a sequence of nonconstant polynomials fulfilling the following conditions:
\begin{enumerate}
  \item[(a)] The sequence \,$(m(n))$ \,of valences is unbounded.
  \item[(b)] At least one of the properties {\rm (P)}, {\rm (Q)}, or {\rm (R)} is satisfied.
\end{enumerate}
Then the set \,$HC((P_n(D)))$ \,is spaceable in \,$H(\C)$.
\end{theorem}

\begin{proof}
Note that, by passing to a subsequence if necessary, we can suppose without loss of generality that the sequence \,$(m(n))$ \, is strictly increasing and even that
\begin{equation*}
  3 \leq m(1) \leq d(1) < m(2) \leq d(2) < m(3) \leq d(3) <\cdots< m(n) \leq d(n) \cdots \to \infty.
\end{equation*}
The space $H(\C)$ is metrizable and separable (see \cite[pp. 370 and 373]{kothe}). Thus, we will check all properties (i) to (iv) in Theorem \ref{Bonet-Martinez-Peris} for \,$X=Y=H(\C)$, \,$T_n=P_n(D)$, and appropriate dense subsets \,$X_0\subset H(\C)$, $Y_0\subset H(\C)$, closed infinite dimensional subspace \,$M_0\subset H(\C)$, and mappings \,$S_n:Y_0\to H(\C)$.

\vskip 3pt

First of all, let us prove that if (a) is satisfied, then (iv) in Theorem \ref{Bonet-Martinez-Peris} holds for a suitable \,$M_0$. For this, we shall extend the nice approach given in \cite[Example 10.13]{grosseP}. Set \,$n_1:=1$ \,and, proceeding recursively, suppose that for some \,$k\in\N$ \,the naturals \,$n_1<n_2<\cdots<n_k$ \,have been selected. Since
\,$\frac{x\cdot \log 2}{\log x}\to +\infty$ \,as \,$x\to +\infty$\, and \,$(m(n))$ \,increases unboundedly, we can choose \,$n_{k+1}\in\N$ \,with \,$n_{k+1}>n_k$ \,and
\begin{equation*}
  \log(A_k) + d(n_k) < \frac{m(n_{k+1})}{\log(m(n_{k+1}))} \cdot \log 2,
\end{equation*}
where
\begin{equation*}
  A_k := \sum_{s=m(n_k)}^{d(n_k)}|c_{n_k,s}|.
\end{equation*}
Thus, a sequence \,$(n_j)$ \,of positive integers has been defined. Note that \,$(n_j)$ \,is a strictly increasing sequence with \,$\log(m(n_j))\geq\log 3 > 1$ \,and that \,$\frac{x\cdot\log 2}{\log x}$ \,is a strictly increasing function on \,$[e,+\infty)$, so on \,$[3,+\infty)$. As a consequence, we get
\begin{equation*}
  \frac{\log(A_k)}{\log(m(n_j))} + d(n_k) < \log(2) \cdot \frac{m(n_j)}{\log(m(n_j))}
\end{equation*}
for all \,$j\geq k+1$, or equivalently,
\begin{equation}\label{Ineq Ak}
  \left(\sum_{s = m(n_k)}^{d(n_k)} |c_{n_k,s}| \right) \cdot m(n_j)^{d(n_k)} < 2^{m(n_j)}
\end{equation}
for all $j\geq k+1$.

\vskip 3pt

Let us define \,$M_0$ \,as the set of all entire functions of the form
\begin{equation}\label{Fs in the subspace M0}
  f(z) = \sum_{j=1}^{\infty} a_j z^{m(n_j)},
\end{equation}
that is, \,$f^{(k)}(0)=0$ \,for all \,$k\notin\left\{m(n_j):j\in\N\right\}$. If a sequence \,$(f_n)$\, in \,$H(\C)$\, converges to a function \,$f$\, uniformly on compacta, then \,$\lim_{n\to\infty}f_n^{(k)}(0)=f^{(k)}(0)$\, for all \,$k\in\N\cup\{0\}$, which implies that \,$M_0$ \,is a closed subspace of \,$H(\C)$. Moreover, since the function \,$f_j(z)=z^{m(n_j)}$\, belongs to \,$M_0$ \,for each \,$j$, it follows that \,$M_0$ \,is infinite-dimensional.

\vskip 3pt

Let us show that \,$P_{n_k}(D)f\to 0$ as $k\to\infty$ \,for every \,$f\in M_0$, which will yield the desired property (iv) in Theorem \ref{Bonet-Martinez-Peris} for the sequence of operators $(P_{n_k}(D))$. Take \,$f\in M_0$ \,as in \eqref{Fs in the subspace M0}. Given a compact subset \,$K \subset \C$, there is \,$r>1$ \,such \,$|z|\leq r$ \,for all \,$z \in K$. For every \,$z\in K$ \,we have
\begin{align*}
  |(P_{n_k}(D)f)(z)|
  &= \left| \sum_{s=m(n_k)}^{d(n_k)} c_{n_k,s} \cdot
     D^s\left(\sum_{j=1}^{\infty} a_j z^{m(n_j)} \right) \right| \\
  &= \left| \sum_{s=m(n_k)}^{d(n_k)} c_{n_k,s} \cdot
     D^s \left( \sum_{j=k}^{\infty} a_j z^{m(n_j)} \right) \right| \\
  &= \left| \sum_{j=k}^{\infty} a_j \sum_{s=m(n_k)}^{d(n_k)} c_{n_k,s}\cdot m(n_j) \cdot
     (m(n_j)-1) \cdots (m(n_j)-s+1) z^{m(n_j)-s} \right| \\
  &\leq \sum_{j=k}^{\infty} |a_j| \left(\sum_{s=m(n_k)}^{d(n_k)}|c_{n_k,s}|\right)
     \cdot m(n_j)^{d(n_k)} \cdot r^{m(n_j)}.
\end{align*}
Then the inequality in \eqref{Ineq Ak} implies that
\begin{equation*}
  |(P_{n_k}(D)f)(z)| \leq \sum_{j=k}^{\infty} |a_j|(2r)^{m(n_j)} \to 0
  \quad \text{as } k\to\infty,
\end{equation*}
where the last limit is due to the fact that \,$f$ \,is entire. This shows that \,$P_{n_k}(D)f\to 0$ uniformly on compacta for all \,$f\in M_0$, as required.

\vskip 3pt

Next we will prove that if the property (P) is satisfied, then (i), (ii), and (iii) in Theorem \ref{Bonet-Martinez-Peris} hold for suitable \,$X_0$, \,$Y_0$, and \,$S_n$ ($n\in\N$). Here, we shall follow the lines given in the proof of Theorem 2.4 in \cite{bernalprado}. Assume that (P) is satisfied, that is, there is an $\mathcal{E}$-unicity set $U\subset\C$ with \,$\lim_{n\to\infty}|P_n(z)|=+\infty$\, for all \,$z\in U$. Define \,$X_0$\, as the set of all polynomials on $\C$ and \,$Y_0:={\rm span}\left\{e_w : \, w\in U\right\}$. For each $n\in\N$, the mapping \,$S_n:Y_0\to H(\C)$ is defined by extending linearly the following one defined on \,$\{e_w : w\in U\}$:
\begin{equation*}
  S_n(e_w)=\begin{cases}
             0                  & \hbox{if } P_n(w)=0 \\
             \frac{e_w}{P_n(w)} & \hbox{otherwise.}
           \end{cases}
\end{equation*}
Let us recall that the sequence \,$(m(n))$ \,strictly increases to infinity. On the one hand, given a polynomial \,$g$, we have \,$D^j g=0$ \,for every \,$j>{\rm degree}\,(g)$. Hence, \,$T_n g=P_n(D)g=0$ \,for \,$n$ \,large enough, which entails \,$T_n g\to 0$ as $n\to\infty$ \,for all \,$g\in X_0$. On the other hand, the property (P) implies that \,$S_n\to 0$ as $n\to\infty$ \,on \,$\{e_w : w\in U\}$, and so on its linear span \,$Y_0$.

\vskip 3pt

For fixed $n\in\N$ and $w\in\C$, the following equalities hold for all $z\in\C$:
\begin{equation*}
  (T_n e_w)(z) = (P_n(D)e_w)(z) = \sum_{j=m(n)}^{d(n)}c_{j,n} D^j(e^{wz})
  = \sum_{j=m(n)}^{d(n)}c_{j,n} w^j e^{wz} = P_n(w)e^{wz}.
\end{equation*}
That is, \,$T_n e_w=P_n(w)e_w$. Hence, for each \,$w\in U$ \,and for \,$n$ \,large enough we get
\begin{equation*}
  T_n S_n (e_w) = T_n\left(\frac{e_w}{P_n(w)}\right) = \frac{P_n(w)e_w}{P_n(w)} = e_w.
\end{equation*}
Therefore, the sequence \,$(T_nS_n)$ \,tends to the identity on \,$\left\{e_w: w\in U\right\}$ and, consequently, on its span \,$Y_0$. This concludes the proof of (i), (ii), and (iii) in Theorem \ref{Bonet-Martinez-Peris} in the case that (P) holds.

\vskip 3pt

Finally, let us assume that (Q) is fulfilled and prove the properties (i), (ii), and (iii) in Theorem \ref{Bonet-Martinez-Peris}. In this case, we shall follow closely the lines given in the proof of Theorem 2.15 in \cite{bernalprado}. Although only a minimal adaptation of the proof is necessary, it will be reproduced for the sake of completeness. This time both \,$X_0$\, and \,$Y_0$\, will be the set of all polynomials on $\C$. As it has been mentioned previously, we have that \,$T_n\to 0$ \,pointwise on \,$X_0$.

\vskip 3pt

For fixed $k\in\N\cup\{0\}$ and $n\in\N$, our aim now is to find a polynomial $f$ such that \, $T_n f(z)=z^k$. In order to do that, we define the coefficient $a_{j,n}$ as follows:
\begin{equation*}
  a_{j,n}=\begin{cases}
            c_{j+m(n),n} & \text{if } 0\leq j\leq d(n)-m(n), \\
            0            & \text{otherwise.}
           \end{cases}
\end{equation*}
In particular, $a_{0,n}=c_{m(n),n}\neq 0$ by the definition of valence. We consider the function
\begin{equation*}
  \Psi_n(z) = \sum_{j=0}^{\infty}a_{j,n}z^j = \sum_{j=0}^{d(n)-m(n)}a_{j,n}z^j
\end{equation*}
and the associated equation
\begin{equation}\label{Eq-1 caso (Q)}
  \Psi_n(D) g(z) = z^k,
\end{equation}
where $g$ will be a polynomial of degree not greater than $k$, say, $g(z)=\sum_{s=0}^{k} b_{s,n}z^s$. That is,
\begin{equation*}
  \sum_{j=0}^{d(n)-m(n)}a_{j,n} \left( \sum_{s=0}^{k} b_{s,n} z^s \right)^{(j)} = z^k.
\end{equation*}
Since \,$g^{(j)}(z)=0$\, for all \,$j>k$\, and \,$a_{j,n}=0$\, for all \,$j>d(n)-m(n)$, the equation \eqref{Eq-1 caso (Q)} is equivalent to
\begin{equation*}
  \sum_{j=0}^{k}a_{j,n} \left( \sum_{s=0}^{k} b_{s,n} z^s \right)^{(j)} = z^k,
\end{equation*}
which in turn is equivalent to the following system:
\begin{equation*}
  \begin{cases}
    \sum_{j=s}^{k} a_{j-s,n} b_{j,n} \cdot \frac{j!}{s!} = 0 & \text{for each } s=0,1,...,k-1 \\
    a_{0,n}b_{k,n} = 1.
  \end{cases}
\end{equation*}
This is a square linear system with determinant $a_{0,n}^{k+1}\neq 0$, so it has a unique solution $(b_{0,n},...,b_{k,n})$. Cramer's rule yields
\begin{equation*}
  b_{s,n} = \frac{1}{a_{0,n}^{k+1}} \cdot \sum_{j=1}^{k} \Phi_{j,s,k}(a_{1,n},...,a_{k,n}) \, a_{0,n}^j
\end{equation*}
for each $s\in\{0,1,...,k\}$, where each $\Phi_{j,s,k}$ is a polynomial of $k$ complex variables not depending on $n$. With these values, the function $g_{n,k}(z)=\sum_{s=0}^{k} b_{s,n}z^s$ satisfies the equation $\Psi_n(D)g_{n,k}=z^k$. Therefore, if the polynomial $f_{n,k}$ is defined as
\begin{equation}\label{Eq-2 caso (Q)}
  f_{n,k}(z) = \sum_{s=0}^{k} b_{s,n} \frac{z^{s+m(n)}}{(s+1)(s+2)\cdots(s+m(n))},
\end{equation}
then
\begin{align*}
  (T_nf_{n,k})(z) &= (P_n(D)f)(z) = \sum_{j=m(n)}^{d(n)}c_{j,n}f_{n,k}^{(j)}(z)
  = \sum_{j=0}^{d(n)-m(n)}a_{j,n}D^j \left(f_{n,k}^{(m(n))}\right)(z) \\
  &= \sum_{j=0}^{d(n)-m(n)}a_{j,n}g_{n,k}^{(j)}(z)
  = (\Psi_n(D)g_{n,k})(z) = z^k.
\end{align*}
This proves that the polynomial $f_{n,k}$ satisfies the equation $T_n f_{n,k}(z)=z^k$.

\vskip 3pt

From (Q), the sequence $\{c_{j+m(n),n}: n\in\N\}$ \,is bounded for each $j\in\N$, so the sequences $(a_{1,n}),(a_{2,n}),\ldots,(a_{k,n})$ are bounded. As a consequence, there exists a finite positive constant $C$, not depending on $n$, such that
\begin{equation*}
  |\Phi_{j,s,k} (a_{1,n},...,a_{k,n})| \leq C
\end{equation*}
for all $s\in\{0,1,...,k\}$ and all $j\in\{1,...,k\}$, which implies
\begin{equation}\label{Eq-3 caso (Q)}
  \left|b_{s,n}\right| \leq \frac{1}{|a_{0,n}|^{k+1}} \cdot \sum_{j=1}^{k}C|a_{0,n}|^j
  = \sum_{j=1}^{k}\frac{C}{|a_{0,n}|^{k+1-j}}.
\end{equation}
Let us fix \,$r>1$. By equations \eqref{Eq-2 caso (Q)} and \eqref{Eq-3 caso (Q)}, if \,$|z|\leq r$, then
\begin{align*}
  |f_{n,k}(z)| &\leq \sum_{s=0}^{k} |b_{s,n}|\cdot \frac{r^{s+m(n)}}{(s+1)(s+2)\cdots(s+m(n))} \\
  &\leq \sum_{s=0}^{k}\sum_{j=1}^{k}\frac{C}{|a_{0,n}|^{k+1-j}} \cdot
  \frac{r^{k+m(n)}}{(s+1)(s+2)\cdots(s+m(n))} \\
  &= \sum_{s=0}^{k}\sum_{j=1}^{k}\frac{C}{|a_{0,n}|^{k+1-j}} \cdot
  \frac{r^{k+m(n)}\cdot s!}{(s+m(n))!}.
\end{align*}
Since \,$p!\cdot q!\leq (p+q)!$ \,for all \,$p,q\in\N$, if \,$|z|\leq r$, we obtain
\begin{equation*}
  |f_{n,k}(z)| \leq \sum_{s=0}^{k}\sum_{j=1}^{k}\frac{C}{|a_{0,n}|^{k+1-j}} \cdot
  \frac{r^{k+m(n)}}{m(n)!}
  = (k+1)\sum_{j=1}^{k}Cr^k \cdot \frac{r^{m(n)}}{|c_{m(n),n}|^{k+1-j} \cdot m(n)!}.
\end{equation*}
Recall that $\lim_{n\to\infty}m(n)=+\infty$ and $\lim_{n\to\infty}m(n)|c_{m(n),n}|^{\frac{k+1-j}{m(n)}}=+\infty$ for every $j$ by the property (Q). Then Stirling's formula leads us to
\begin{equation*}
  \lim_{n\to\infty}\Big(|c_{m(n),n}|^{k+1-j} \cdot m(n)!\Big)^{\frac{1}{m(n)}} = +\infty,
\end{equation*}
so there is $n_0\in\N$ such that
\begin{equation*}
  \Big(|c_{m(n),n}|^{k+1-j} \cdot m(n)!\Big)^{\frac{1}{m(n)}} > 2r
\end{equation*}
for all $n\geq n_0$. Hence, if $|z|\leq r$ and $n\geq n_0$, then
\begin{equation*}
  |f_{n,k}(z)| \leq (k+1)\sum_{j=1}^{k}Cr^k \cdot \frac{r^{m(n)}}{(2r)^{m(n)}} \to 0
  \quad \text{as } n\to\infty.
\end{equation*}
Therefore, for each $k\in\N\cup\{0\}$, $f_{n,k}\to 0$ uniformly on compacta as $n\to\infty$.

\vskip 3pt

The mapping $S_n:Y_0\to H(\C)$ is defined on the space \,$Y_0$ \,of all polynomials by
\begin{equation*}
  S_n\left(z^k\right) = f_{n,k}(z)
\end{equation*}
for each $k\in\N\cup\{0\}$ \,and\, $n\in\N$ and then extended to \,$Y_0$ \,by linearity. We have already proved that, for each $k$, $S_n \left(z^k\right)=f_{n,k}\to 0$ as $n\to\infty$ and $T_n\left(S_n z^k\right)=z^k$. Hence, it is plain that \,$S_n f\to 0$ \,and \,$T_n(S_n f)=f\to f$ \,as \,$n\to\infty$ \,for every polynomial \,$f\in Y_0$. Consequently, properties (i)-(ii)-(iii) in Theorem \ref{Bonet-Martinez-Peris} are satisfied also in this case. This concludes the proof.
\end{proof}

\vskip 3pt

Notice that, for any nonconstant polynomial \,$P$ \,with \,$P(0)=0$, the valences of the sequence of polynomials \,$(P(z)^n)$ \,are unbounded. Moreover, since \,$|P(z)|\to\infty$ \,as \,$|z|\to\infty$, it follows that \,$|P(z)|>2$ \,on some circle \,$|z|=r$ \,and then the property (R) is satisfied. Hence, Menet's result of spaceability of \,$HC(P(D))$ \,is recovered in this case, since \,$(P^n)(D)=P(D)\circ \cdots \circ P(D)$ ($n$ times) \,for all \,$n \in \N$.

\begin{theorem}\label{max-dens-lineable}
Let \,$\left\{P_n(z)=\displaystyle{\sum_{j=m(n)}^{d(n)}c_{j,n} z^j}\right\}_{n\geq 1}$ \,be a sequence of nonconstant polynomials fulfilling the following conditions:
\begin{enumerate}
  \item[(a)] The sequence \,$(m(n))$ \,of valences is unbounded.
  \item[(b)] At least one of the properties \,{\rm (P), (Q) or (R)} \,is satisfied.
\end{enumerate}
Then the set \,$HC((P_n(D)))$ \,is maximal dense-lineable in \,$H(\C)$.
\end{theorem}

\begin{proof}
By Theorem \ref{main theorem spaceable}, there exists an infinite dimensional closed subspace $M\subset HC((P_n(D)))\cup\{0\}$. By \cite{Popoola}, the dimension of both $M$ and $H(\C)$ is $\mathfrak{c}$. Then the ma\-xi\-mal dense-lineability of \,$HC((P_n(D)))$ \,follows from Theorem \ref{A stronger than B} with the following choice of characters: $X=H(\C)$, $A=HC((P_n(D)))$, and $B$ the set of all polynomials on $\C$. Indeed, since the polynomials form a dense subspace of \,$H(\C)$, we obtain the dense-lineability of \,$B$. Now, for a prescribed polynomial \,$g$ \,there is \,$n_0\in\N$ \,such that \,$m(n)>{\rm degree}\,(g)$ \,for all \,$n\geq n_0$, so \,$D^j g=0$ \,for all \,$j\geq m(n)$ \,if \,$n\geq n_0$. Thus, $P_n(D)g=0$ \,for every \,$n\geq n_0$. Therefore, no polynomial can be hypercyclic, that is, $A\cap B=\varnothing$. Finally, let \,$f\in A$, $g\in B$, \,and \,$h\in H(\C)$. Since $f$ is hypercyclic, there is a strictly increasing sequence \,$(n_k)\subset\N$ \,such that \,$P_{n_k}(D)f\to h$ \, uniformly on compacta. As \,$P_n(D)g\to 0$ \,as \,$n\to\infty$, we obtain that $P_{n_k}(D)(f+g)\to h$ as $k\to\infty$ as well. In other words, $f+g$ \,is hypercyclic or, with language of sets, $A+B\subset A$. This completes the checking of all properties (i) to (iv) in Theorem \ref{A stronger than B}.
\end{proof}

\vskip 3pt

\section{Subspaces with prefixed hypercyclic functions.}\label{SeccionPuntual}

In \cite{PellegrinoRaposo}, Pellegrino and Raposo introduced the notion of pointwise lineable set, which in turn led to the concept of infinitely pointwise lineable set defined in \cite{CalderonGerlachPrado}. Given a cardinal number $\alpha$, a subset $A$ of a topological vector space $X$ is said to be {\it pointwise $\alpha$-lineable} if for every $x\in A$ there exists a vector subspace $W_x\subset X$ such that $\dim(W_x)=\alpha$ and $x\in W_x\subset A\cup\{0\}$. If each $W_x$ can be chosen to be closed, then $A$ is said to be {\it pointwise $\alpha$-spaceable}. In addition, if $\lambda$ is an infinite cardinal number, the subset $A$ is {\it $\lambda$-infinitely pointwise $\alpha$-dense-lineable} if for every $x\in A$ there is a family $\left\{W_k:k\in\Lambda\right\}$, where $\Lambda$ is set of cardinality $\lambda$, such that each $W_k$ is a dense subspace of $X$, $\dim(W_k)=\alpha$, $x\in W_k\subset A\cup\{0\}$, and $W_k\cap W_l={\rm span}\{x\}$ whenever $k\neq l$. When $\lambda=\aleph_0$, we recover the original definition given in \cite{CalderonGerlachPrado} (see also \cite{Emerick} for an arbitrary $\lambda$).

\begin{theorem}\label{Pointwise spaceability}
Let us suppose that $X$ and \,$Y$ are two topological vector spaces, \,$Y$ is metrizable and separable, and $T_n:X\to Y$ is a continuous linear mapping for each $n\in\N$. Let $\alpha$ be an infinite cardinal number and assume that, for each subsequence $(T_{n_k})$ of $(T_n)$, the set $HC((T_{n_k}))$ is $\alpha$-spaceable. Then $HC((T_n))$ is pointwise $\alpha$-spaceable.
\end{theorem}

\begin{proof}
Plainly, there exists an element \,$x_0\in HC((T_n))$. Then there exists a subsequence \,$(T_{n_k})$\, of \,$(T_n)$\, such that \,$T_{n_k}x_0\to 0$\, as \,$k\to\infty$. From the assumption, the set \,$HC((T_{n_k}))$\, is $\alpha$-spaceable. That is, there is a closed subspace \,$V\subset X$\, such that \,$\dim(V)=\alpha$\, and \,$V\subset HC((T_{n_k}))\cup\{0\}$. Since \,$\left\{T_n x_0:n\in\N\right\}$\, is dense in \,$Y$\, and \,$T_{n_k}x_0\to 0$, it follows that \,$x_0\notin HC((T_{n_k}))\cup\{0\}$\, and, consequently, \,$x_0\notin V$. Since \,$V$\, is closed, the direct sum \,$W=V\bigoplus{\rm span}\{x_0\}$\, is also closed (see, e.g., \cite[p. 32]{Rudin}). Thus, \,$W$\, is a closed subspace of \,$X$\, such that \,$\dim(W)=\alpha$\, and \,$x_0\in W$.

\vskip 3pt

Our final task is to show that $W\subset HC((T_n))\cup\{0\}$. With this aim, let us fixed an element $w\in W\backslash\{0\}$. Then there are a vector $v\in V$ and a scalar $\lambda$ such that $w=v+\lambda x_0$. At this point, we have to distinguish two cases:
\begin{itemize}
  \item If $v=0$, then $\lambda\neq 0$ and hence $w=\lambda x_0\in HC((T_n))$.
  \item Let us suppose now that $v\neq 0$ and let us prove that $w$ is hypercyclic for the sequence $(T_n)$. For a fixed vector $y\in Y$, since $v\in V\backslash\{0\}\subset HC((T_{n_k}))$, there exists a subsequence $\left(T_{n_{k_j}}\right)$ of $(T_{n_k})$ such that $T_{n_{k_j}}v\to y$ as $j\to\infty$. As $T_{n_k}x_0\to 0$ as $k\to\infty$, we also have that $T_{n_{k_j}}x_0\to 0$ as $j\to\infty$. Therefore, by linearity, we get that
      \begin{equation*}
        T_{n_{k_j}}w = T_{n_{k_j}}v + \lambda T_{n_{k_j}}x_0 \to y + 0 = y
        \quad \text{as } j\to\infty.
      \end{equation*}
      This shows that $w\in HC((T_{n_k}))$ in this case too.
\end{itemize}
Summarizing, any vector in $W\backslash\{0\}$ is hypercyclic for the sequence $(T_n)$. Thus, the proof is finished.
\end{proof}

The following assertion, that reinforces Theorem \ref{main theorem spaceable}, is an immediate consequence of Theorems \ref{main theorem spaceable} and \ref{Pointwise spaceability}, since the dimension of every infinite dimensional closed subspace of $H(\C)$ is $\mathfrak{c}$ (see \cite{Popoola}). Note also that, in order to apply Theorem \ref{Pointwise spaceability}, we may start with any subsequence \,$(P_{n_k}(D))$ \,of \,$(P_n(D))$ \,satisfying that \,$m(n_k)$ is increasing to infinity.

\begin{theorem}\label{HC((P_n(D))) es pointwise spaceable}
Let \,$\left\{P_n(z)=\displaystyle{\sum_{j=m(n)}^{d(n)}c_{j,n} z^j}\right\}_{n\geq 1}$ \,be a sequence of nonconstant polynomials fulfilling the following conditions:
\begin{enumerate}
  \item[(a)] The sequence \,$(m(n))$ \,of valences is unbounded.
  \item[(b)] At least one of the properties {\rm (P)}, {\rm (Q)}, or {\rm (R)} \,is satisfied.
\end{enumerate}
Then the set \,$HC((P_n(D)))$ is pointwise $\mathfrak{c}$-spaceable in \,$H(\C)$.
\end{theorem}

Observe that the last result is optimal in terms of the dimension of the closed subspace.

Inspired by \cite[Theorem 2.3]{CalderonGerlachPrado} (see also \cite{Emerick}), we obtain the next pointwise version of Theorem \ref{A stronger than B}:

\begin{theorem}\label{A stronger than B. Pointwise}
Let us suppose that \,$X$ is a metrizable separable topological vector space, $\alpha$ is an infinite cardinal number, and \,$A$ and \,$B$ are two subsets of \,$X$ with the following properties:
\begin{enumerate}
  \item[(ii)] $A+B\subset A$.
  \item[(ii)] $A\cap B=\varnothing$.
  \item[(iii)] $A$ is pointwise $\alpha$-lineable.
  \item[(iv)] $B$ is dense-lineable.
\end{enumerate}
Then the set \,$A$ is $\alpha$-infinitely pointwise $\alpha$-dense-lineable.
\end{theorem}

\begin{proof}
It is essentially the same proof done in \cite[Theorem 2.3]{CalderonGerlachPrado} for a countable collection of subspaces. We will only indicate the part that is different. Let us fix an element $x\in A$. Note that $x$ cannot be $0$. Indeed, if $x=0$, it would follow that
\begin{equation*}
  B = \{0\}+B \subset A+B \subset A,
\end{equation*}
so $B\subset A$, which is impossible because $A\cap B=\varnothing$. Since $A$ is pointwise $\alpha$-lineable, there is a vector space $W\subset A\cup\{0\}$ such that $\dim(W)=\alpha$ and $x\in W$. Let \,$\left\{e_{\gamma}:\gamma\in\Gamma\right\}$ \,be chosen so that \,$\{x\}\cup\left\{e_{\gamma} : \gamma\in\Gamma\right\}$ is a basis of $W$.

\vskip 3pt

The cardinal number $\alpha$ is infinite, so $\alpha^2=\alpha$ and then it is possible to split \,$\Gamma$ \,into $\alpha$ subsets, each one with cardinality \,$\alpha$. That is, $\Gamma=\bigcup_{k\in\mathcal{K}}\Gamma_k$, where the cardinality of the index set \,$\mathcal{K}$ \,is \,$\alpha$, the cardinality of each \,$\Gamma_k$ is $\alpha$, and $\Gamma_k\cap \Gamma_l = \varnothing$ if $k\neq l$. For each \,$k\in\mathcal{K}$, we define
\begin{equation*}
  W_k = {\rm span} \left(\{x\} \cup \left\{e_{\gamma}:\gamma\in\Gamma_k\right\}\right).
\end{equation*}
Note that $x\in W_k$, $\dim(W_k)=\alpha$, and $W_k\cap W_l={\rm span}\{x\}$ if $k\neq l$. In addition, we have
\begin{equation*}
  W_k + W_l \subset W \subset A\cup\{0\}.
\end{equation*}

The basis of each $W_k$ can be renamed in the following way:
\begin{equation*}
  \{x\} \cup \left\{e_{\gamma}:\gamma\in\Gamma_k\right\} = \left\{\omega_i^{(k)}:i\in I\right\},
\end{equation*}
where $I$ is a index set with cardinality $\alpha$ such that $1\in I$ and we assume that $\omega_1^{(k)}=x$. Now we can split $I$ into infinitely many pairwise disjoint nonempty sets: $I=\bigcup_{n\in\N}I_n$, where $1\in I_1$ and $I_n\cap I_m=\varnothing$ if $n\neq m$. From this point, the proof continues exactly as in \cite[Theorem 2.3]{CalderonGerlachPrado}.
\end{proof}

The following assertion is an improvement of Theorem \ref{max-dens-lineable}.

\begin{theorem}
Let \,$\left\{P_n(z)=\displaystyle{\sum_{j=m(n)}^{d(n)}c_{j,n} z^j}\right\}_{n\geq 1}$ \,be a sequence of nonconstant polynomials fulfilling the following conditions:
\begin{enumerate}
  \item[(a)] The sequence \,$(m(n))$ \,of valences is unbounded.
  \item[(b)] At least one of the properties \,{\rm (P), (Q) or (R)} \,is satisfied.
\end{enumerate}
Then the set \,$HC((P_n(D)))$ is $\mathfrak{c}$-infinitely pointwise $\mathfrak{c}$-dense-lineable in \,$H(\C)$.
\end{theorem}

\begin{proof}
It is a consequence of Theorem \ref{A stronger than B. Pointwise} with $A=HC((P_n(D)))$ and $B$ the set of all polynomials on $\C$. It was shown in the proof of Theorem \ref{max-dens-lineable} that $A+B\subset A$ and $A\cap B=\varnothing$. Moreover, $A$ is pointwise $\mathfrak{c}$-lineable by Theorem \ref{HC((P_n(D))) es pointwise spaceable}, while $B$ is dense-lineable because the polynomials form a dense subspace of \,$H(\C)$.
\end{proof}

\vskip 3pt

We want to finish this section by posing the following question, which arises naturally.

\vskip 3pt

\noindent {\bf Question.} Under similar assumptions to those of Theorems \ref{main theorem spaceable} and \ref{max-dens-lineable} for a sequence \,$(\Phi_n)$ of entire functions with exponential type, is the set \,$HC((\Phi_n(D)))$ \,spaceable, maximal dense-lineable, and even infinitely pointwise $\mathfrak{c}$-dense-lineable and pointwise spaceable, in $H(\C)$?

\vskip 3pt

\noindent {\bf Acknowledgments.} This paper has been made during a stay of the third author at the Instituto de Matem\'aticas de la Universidad de Sevilla (IMUS). The third author wants to thank this institution for its support and hospitality.

\vskip 3pt

\noindent {\bf Authors contribution.} All four authors have contributed equally to this work.

\end{document}